\documentclass[12pt,reqno]{amsart}
\usepackage{amsmath}
\usepackage{amsthm}
\usepackage{amsfonts}
\usepackage{amssymb}
\usepackage{mathrsfs}
\usepackage[margin=1.35in, top =1.25in, bottom = 1.25in]{geometry}
\usepackage{setspace}
\usepackage{tikz}
\usetikzlibrary{matrix, arrows}
\usepackage{enumitem}
\usepackage[colorlinks, breaklinks]{hyperref}
\hypersetup{linkcolor=black,citecolor=blue,filecolor=black,urlcolor=black} 
\usepackage{pdflscape}
\usepackage{array}
\usepackage{adjustbox}
\usepackage{tabularx}
\usepackage{multirow}
\usepackage{multicol}
\usepackage{float}

\newcommand{\abs}[1]{\left\vert #1 \right\vert}

\renewcommand{\bar}[1]{\overline{#1}}
\DeclareMathOperator{\codim}{codim}
\renewcommand{\epsilon}{\varepsilon}
\DeclareMathOperator{\Ext}{Ext}
\newcommand{\iso}{\cong}
\DeclareMathOperator{\Hom}{Hom}
\DeclareMathOperator{\hgt}{ht}				% The height of an ideal.
\newcommand{\kk}{k}					% The ground field.
	
\newcommand{\NN}{\mathbb{N}}
\DeclareMathOperator{\Pf}{Pf}				% Ideal of submaximal pfaffians.
\renewcommand{\phi}{\varphi}
\DeclareMathOperator{\pd}{pd}				% The projective dimension of a module.
\newcommand{\QQ}{\mathbb{Q}}
\DeclareMathOperator{\reg}{reg}			% The regularity of a module.
\DeclareMathOperator{\syz}{Syz}			% Syzygy module.
\DeclareMathOperator{\Tor}{Tor}

\newcommand{\ZZ}{\mathbb{Z}}

\newtheorem{thm}{Theorem}[section]
\newtheorem*{thm*}{Theorem}

\newtheorem{prop}[thm]{Proposition}
\newtheorem{cor}[thm]{Corollary}

\newtheorem*{main-thm}{Main Theorem}

\theoremstyle{definition}

\newtheorem{example}[thm]{Example}
\newtheorem{rmk}[thm]{Remark}
\newtheorem{question}[thm]{Question}

\newtheorem*{notation}{Notation}
\newtheorem*{ack}{Acknowledgements}

\numberwithin{equation}{section}
\numberwithin{table}{section}

\begin{document}

\title{Quadratic Gorenstein rings and the Koszul property II}
\date{}

\author{Matthew Mastroeni}
\address{Department of Mathematics, Oklahoma State University, Stillwater, OK 74078}
\email{mmastro@okstate.edu}

\author{Hal Schenck}
\thanks{Schenck supported by NSF 1818646.}
\address{Department of Mathematics, Iowa State University, Ames, IA 50011}
\email{hschenck@iastate.edu}

\author{Mike Stillman}
\thanks{Stillman supported by NSF 1502294.}
\address{Department of Mathematics, Cornell University, Ithaca, NY 14850}
\email{mike@math.cornell.edu}


\maketitle

\begin{abstract}
In \cite{Grobner:flags:and:Gorenstein:algebras}, Conca-Rossi-Valla ask if every quadratic Gorenstein ring $R$ of regularity three is Koszul. In \cite{QGRKP1} we use idealization to answer their question, proving that in nine or more variables there exist quadratic Gorenstein rings of regularity three which are not Koszul. In this paper, we study the analog of the Conca-Rossi-Valla question when the regularity of $R$ is four or more. Let $R$ be a quadratic Gorenstein ring having $\codim R = c$ and $\reg R = r \ge 4$. We prove that if $c = r+1$ then $R$ is always Koszul, and for every $c \geq r+2$, we construct quadratic Gorenstein rings that are not Koszul, answering questions of Matsuda \cite{Matsuda} and Migliore-Nagel \cite{Gorenstein:algebras:presented:by:quadrics}.
\end{abstract}

\begin{spacing}{1.15}
\section{Introduction}\label{intro}

Let $I$ be a homogeneous ideal generated by quadrics in a standard graded polynomial ring $S$ over a field, and set $R=S/I$. In \cite{Tate}, Tate showed that if $I$ is a complete intersection, then $R$ is Koszul. Complete intersections are the simplest examples of Gorenstein rings, and Conca-Rossi-Valla show in \cite{Grobner:flags:and:Gorenstein:algebras} that quadratic Gorenstein rings of regularity two are always Koszul. On the other hand, quadratic Gorenstein rings of regularity three are less well understood:
\begin{itemize}
\item Vishik-Finkelberg \cite{Vishik:Finkelberg} and Polishchuk \cite{Polishchuk} show that the homogeneous
  coordinate ring of a canonical curve generated by quadrics is Koszul.
\item Conca-Rossi-Valla  \cite{Grobner:flags:and:Gorenstein:algebras} show that if $R$ has codimension at most four
  it is Koszul.
\item Caviglia \cite{Caviglia} shows that if $R$ has codimension five it is Koszul.
\end{itemize}
Conca-Rossi-Valla ask in \cite{Grobner:flags:and:Gorenstein:algebras} if it is possible that all quadratic Gorenstein rings of regularity three are Koszul.  In \cite{QGRKP1} we negatively answer this question; using Nagata's technique of idealization, we produce quadratic Gorenstein rings of regularity three which are not Koszul for all codimensions $c \geq  9$.

Our work was motivated by Matsuda's discovery in \cite{Matsuda} of a quadratic Gorenstein ring with regularity four and codimension seven which is not Koszul. Matsuda constructs his example via graph theory; his methods do not produce further examples of quadratic Gorenstein rings which are not Koszul. Matsuda's example is not subsumed by the results of \cite{QGRKP1}, as it does not arise as an idealization (however, see Remark \ref{generalizing:Matsuda}). 

Tensoring a quadratic Gorenstein ring $R$ with appropriate choices of $R'$ yields new quadratic Gorenstein rings $R\otimes_kR'$. Applying this to non-Koszul quadratic Gorenstein rings $R$ constructed by idealization and by Matsuda, we show in \cite[4.9]{QGRKP1} that quadratic Gorenstein rings and the Koszul property are related to each other in characteristic zero\footnote{The characteristic assumption is needed only for a small number of examples verified using {\tt Macaulay2} \cite{Macaulay2}.}  by the table appearing in Figure \ref{table:of:when:quadratic:Gorenstein:rings:are:Koszul}. Our goal in this paper is to address the lacunae in the table.

\begin{figure}[h]
\hspace{3 em}
\begin{tikzpicture}[scale = 0.56]
% AXES -------------------------------------------------------------------
\draw[gray, very thin] (0, 0.1) grid (12.9, 13);
\draw[->] (-1,13) -- (13,13);
\draw[->] (0, 14) -- (0, 0); 
\node[right] (c) at (13, 13) {$c$};
\node[below] (r) at (0, 0) {$r$};

% POINTS -------------------------------------------------------------------
\foreach \c in {0,1,2,3,4,5,6,7,8,9,10,11} {
	\node[above] (\c) at (\c+1, 13) {\tiny \c};
	\node[left] (\c +12) at (0, 12-\c) {\tiny \c};
	};
	
% Yes	
\foreach \c/\r in {0/0,1/1,2/2,3/3,4/4,5/5,6/6,7/7,8/8,9/9,10/10,11/11,3/2,4/2,5/2,6/2,7/2,8/2,9/2,10/2,11/2,4/3,5/3} {
	\draw[thick, fill = black!75] (\c + 1, 12-\r) circle (0.25);
	};
	
% Maybe	
\foreach \c/\r in {6/3,7/3,8/3,6/4,8/4,9/4,7/5, 8/6,9/7,10/8,11/9, 9/5,10/5,5/4,6/5,7/6,8/7,9/8,10/9,11/10} {
	\fill[black!75] (\c+1, 11.75 - \r) -- (\c+1, 12.25 - \r) arc (90:270:0.25) -- cycle;
	\draw[thick] (\c + 1, 12 - \r) circle (0.25);
	};
	
%No	
\foreach \c/\r in {9/3,10/3,11/3,10/4,11/4,11/5,7/4,8/5,9/6,10/7,11/8,10/6,11/7,11/6} {
	\draw[thick] (\c + 1, 12-\r) circle (0.25);
	};
	
% LEGEND -------------------------------------------------------------------
\draw[thick, fill = black!75] (15, 8) circle (0.25);
\node[right] (y) at (15.5, 8) {\small Yes};
\draw[thick] (15, 7) circle (0.25);
\node[right] (n) at (15.5, 7) {\small No};
\fill[black!75] (15, 5.75) -- (15, 6.25) arc (90:270:0.25) -- cycle;
\draw[thick] (15, 6) circle (0.25);
\node[right] (m) at (15.5, 6) {\small Unknown};
\end{tikzpicture}
\caption{\small{Is every quadratic Gorenstein ring of codimension $c$ and regularity $r$ Koszul?}}
\label{table:of:when:quadratic:Gorenstein:rings:are:Koszul}
\end{figure}
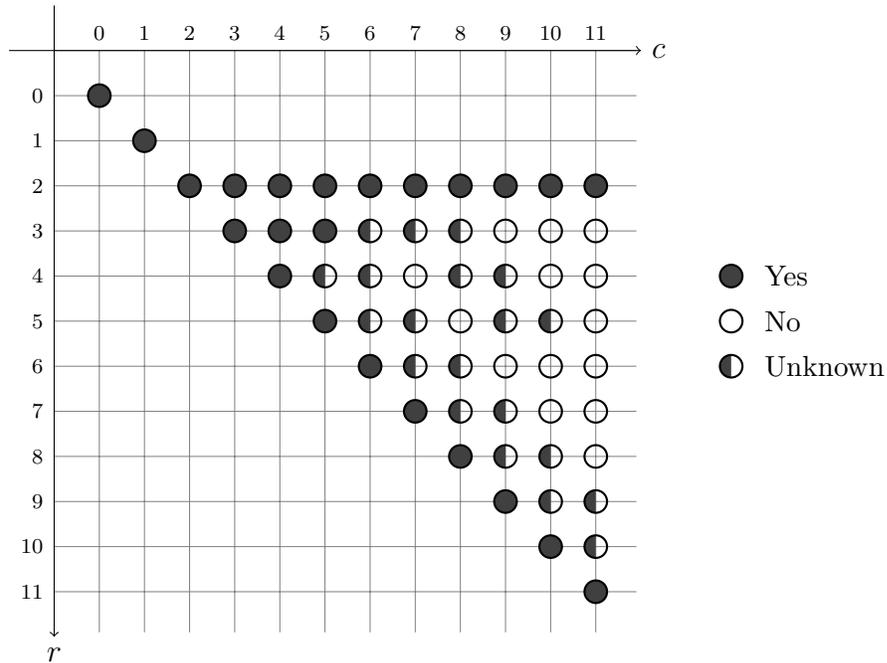

\begin{notation} 
Throughout this paper, unless specifically stated otherwise, we will denote by $\kk$ a fixed ground field of arbitrary characteristic, by $S$ a standard graded polynomial ring over $\kk$, by $I \subseteq S$ a proper graded ideal, and $R = S/I$. Recall that $I$ is \emph{nondegenerate}\index{nondegenerate} if it contains no linear forms; for arbitrary $I$ we can killing a basis for the linear forms contained in $I$ to reduce to the nondegenerate setting, and we will assume that this is the case throughout.  We set $c = \codim R = \hgt I$ and $r = \reg R$.
\end{notation}

\subsection{Key players}
In this paper, we study the relationship between two conditions that impose extraordinary constraints on the homological properties of $R$, namely the Gorenstein and Koszul properties.  The ring $R$ is \emph{Gorenstein} if it is Cohen-Macaulay and its canonical module is isomorphic to a shift of $R$:
\[
\omega_R = \Ext_S^c(R,S)(-n) \iso R(a)
\]
where $\dim S = n$. This implies that the \emph{graded Betti numbers}
\[
\beta_{i,j}^S(R) = \dim_{\kk} \Tor_i^S(R,\kk)_j
\]
have a symmetry 
\begin{equation} \label{Gorenstein:Betti:table:symmetry}
\beta_{i,j}^S(R) = \beta^S_{c-i,c+r-j}(R)
\end{equation}
for all $i, j$.  This information is compactly summarized in the \emph{Betti table} of $R$, where the entry in column $i$ and row $j$ is $\beta_{i,i+j}^S(R)$ (see below for an example). We also recall that the \emph{regularity} of $R$ is 
\[
\reg R = \max\{j \mid \beta_{i,i+j}^S(R) \neq 0 \; \text{for some}\; i \}
\]
It is the index of the bottom-most nonzero row in the Betti table of $R$. When $R$ is Cohen-Macaulay, the regularity is also the length of the \emph{$h$-vector} $h(R) = (1, c, h_2, \dots, h_r)$, which is related to the Hilbert series of $R$ by 
\[ H_R(t) = \frac{1 + ct + h_2t^2 + \cdots + h_rt^r}{(1-t)^{\dim R}} \]
It is well-known that the $h$-vector of a Gorenstein ring is also symmetric in the sense that $h_i = h_{r-i}$ for all $i$, and the regularity $r$ is related to the shift $a$ in the canonical module via $r=a+n-c$.

On the other hand, $R$ is \emph{Koszul} if the ground field $R/R_+ \iso \kk$ has a linear free resolution over $R$.  That is, we have $\beta^R_{i,j}(\kk) = 0$ for all $i$ and $j$ with $j \neq i$.  Koszul algebras have strong duality properties and appear as many rings of interest in commutative algebra,  topology, and algebraic geometry; see the surveys \cite{Froberg:Koszul:algebras:survey} and \cite{Koszul:algebras:and:their:syzygies} and the references therein. A necessary condition for $R$ to be Koszul is that the ideal $I$ is generated by quadrics; however, this is not sufficient.

\begin{example}[{\cite[1.3]{Matsuda}}]  \label{Matsuda's:example}
Matsuda constructs a toric sevenfold in $\mathbb{P}^{14}$ whose coordinate ring $R$ is quadratic and Gorenstein but not Koszul.  Its Betti table appears below. 
\begin{center}
\scalebox{0.835}{
\begin{tabular}{r|cccccccc}
 & 0 & 1 & 2 & 3 & 4 & 5 & 6 & 7  \\
\hline  
0 & 1 & -- & -- & -- & -- & -- & -- & -- 
\\ 
1 & -- & 14 & 21 & -- & -- & -- & -- & --
\\ 
2 & -- & -- & 36 & 126 & 126 & 36 & -- & --
\\ 
3 & -- & -- & -- & -- & -- & 21 & 14 & -- 
\\ 
4 & -- & -- & -- & -- & -- & -- & -- & 1 
\end{tabular}
}
\end{center} 
A {\tt Macaulay2} computation shows that $\beta_{3,4}^R(\kk) = 1$, so $R$ is not Koszul.  Matsuda builds the ideal from a certain graph, but it does not
generalize in any obvious fashion. The $h$-vector of $R$ is $(1, 7, 14, 7, 1)$.
\end{example}

\begin{question}[{\cite[2.1]{Matsuda}}] \label{beating:Matsuda:in:regularity:4}
Do there exist non-Koszul quadratic Gorenstein rings with $r = 4$ and $c \leq 6$?
\end{question}

The $h$-vector of a quadratic Gorenstein ring with $r = 4$ and $c \leq 6$ which is not a complete intersection necessarily has the form $(1, c, h_2, c, 1)$, where $c = 5, 6$.  In \cite{Gorenstein:algebras:presented:by:quadrics}, Migliore-Nagel show that the only possible $h$-vectors are 
\[
\begin{array}{cccc}
  (1, 5, 8, 5, 1) & (1, 6, 10, 6, 1) & (1, 6, 11, 6, 1) & (1, 6, 12, 6, 1)
\end{array}
\]
but they were unable to construct an example with $h_2 = 12$ in characteristic zero and ask:

\begin{question} \label{Migliore-Nagel:h-vector:exists}
Do there exist quadratic Gorenstein rings of characteristic zero with $h$-vector $(1, 6, 12, 6, 1)$?
\end{question}

\subsection{Results} Theorem~\ref{c=r+1:is:Koszul} shows that any quadratic Gorenstein ring $R$ with $c = r + 1$ is always Koszul.  In particular, any quadratic Gorenstein with $h$-vector $(1,5,8,5,1)$ is always Koszul. However, in Example~\ref{r=4:c=6}, we construct a non-Koszul quadratic Gorenstein ring over $\QQ$ with Hilbert function $(1,6,12,6,1)$, affirmatively answering Questions \ref{beating:Matsuda:in:regularity:4} and \ref{Migliore-Nagel:h-vector:exists}. Our results are
summarized in the improvement of Figure \ref{table:of:when:quadratic:Gorenstein:rings:are:Koszul} below.

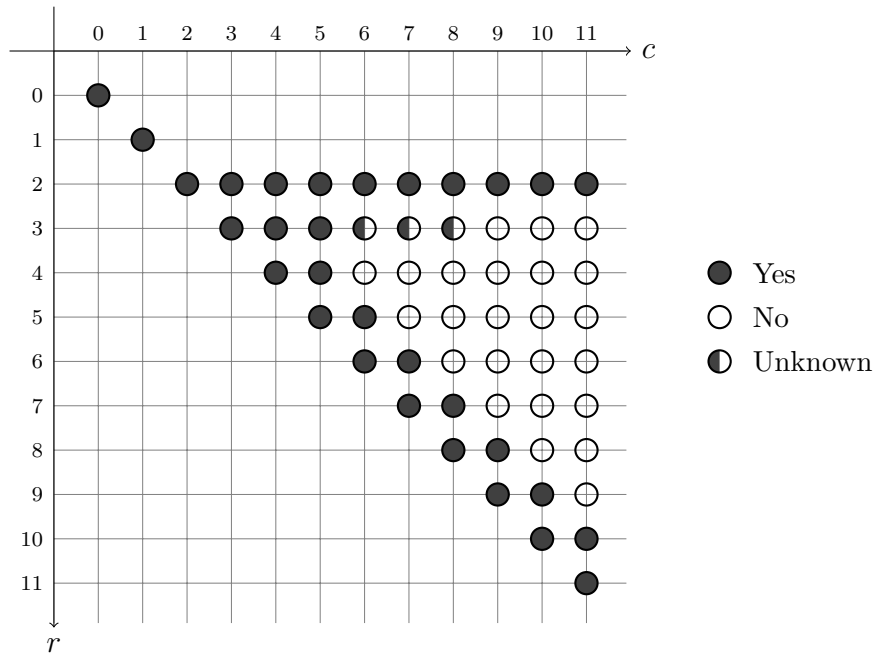
\begin{figure}[H]
\hspace{3 em}
\begin{tikzpicture}[scale = 0.55]
% AXES -------------------------------------------------------------------
\draw[gray, very thin] (0, 0.1) grid (12.9, 13);
\draw[->] (-1,13) -- (13,13);
\draw[->] (0, 14) -- (0, 0); 
\node[right] (c) at (13, 13) {$c$};
\node[below] (r) at (0, 0) {$r$};

% POINTS -------------------------------------------------------------------
\foreach \c in {0,1,2,3,4,5,6,7,8,9,10,11} {
	\node[above] (\c) at (\c+1, 13) {\tiny \c};
	\node[left] (\c +12) at (0, 12-\c) {\tiny \c};
	};
	
% Yes	
\foreach \c/\r in
{0/0,1/1,2/2,3/3,4/4,5/5,6/6,7/7,8/8,9/9,10/10,11/11,3/2,4/2,5/2,6/2,7/2,8/2,9/2,10/2,11/2,4/3,5/3,5/4,6/5,7/6,8/7,9/8,10/9,11/10} {
	\draw[thick, fill = black!75] (\c + 1, 12-\r) circle (0.25);
	};
	
% Maybe	
\foreach \c/\r in {6/3,7/3,8/3}{
	\fill[black!75] (\c+1, 11.75 - \r) -- (\c+1, 12.25 - \r) arc (90:270:0.25) -- cycle;
	\draw[thick] (\c + 1, 12 - \r) circle (0.25);

	};
	
%No	
\foreach \c/\r in {9/3,10/3,11/3,10/4,11/4,11/5,7/4,8/5,9/6,10/7,11/8,10/6,11/7,11/6,6/4,8/4,9/4,7/5, 8/6,9/7,10/8,11/9, 9/5,10/5} {
	\draw[thick] (\c + 1, 12-\r) circle (0.25);
	};
	
% LEGEND -------------------------------------------------------------------
\draw[thick, fill = black!75] (15, 8) circle (0.25);
\node[right] (y) at (15.5, 8) {\small Yes};
\draw[thick] (15, 7) circle (0.25);
\node[right] (n) at (15.5, 7) {\small No};
\fill[black!75] (15, 5.75) -- (15, 6.25) arc (90:270:0.25) -- cycle;
\draw[thick] (15, 6) circle (0.25);
\node[right] (m) at (15.5, 6) {\small Unknown};
\end{tikzpicture}
\caption{\small{Is every quadratic Gorenstein ring of codimension $c$ and regularity $r$ Koszul?}}
\label{improved:table:of:when:quadratic:Gorenstein:rings:are:Koszul}
\end{figure}

The division of the rest of the paper is as follows.  We start with a review of the relevant properties of quadratic Gorenstein rings in \S \ref{background}. In \S \ref{deviation:2:case}, we prove: 

\begin{thm*} 
If $R = S/I$ is a quadratic Gorenstein ring over an infinite ground field with $c = r + 1$, then 
\[I = \Pf(M) + (q_6, \dots, q_{c+2})\] 
for some $5 \times 5$ alternating matrix $M$ of linear forms such that $\hgt \Pf(M) = 3$ and quadrics $q_i$ which form a regular sequence modulo $\Pf(M)$.
\end{thm*}

As a consequence, any quadratic Gorenstein ring $R$ with $c = r + 1$ is a Koszul algebra.  In \S \ref{r>=4}, we use inverse systems to prove:

\begin{thm*}
There is a family of quadratic Gorenstein non-Koszul algebras of regularity four with Hilbert function
$(1,c,2c,c,1)$ for every codimension $c \ge 7$.
\end{thm*}

We close by constructing a regularity-four quadratic Gorenstein non-Koszul algebra over $\QQ$ with Hilbert function $(1,6,12,6,1)$ related to the family of the previous theorem. This answers the questions of \cite{Matsuda} and \cite{Gorenstein:algebras:presented:by:quadrics}.

\section{Background on Quadratic Gorenstein Rings} \label{background}

The mains tools that we will use in the following sections are linkage and inverse systems.  We briefly review some terminology and establish our conventions regarding these topics below, and we refer the reader to \cite[Ch.~21]{Eisenbud},\cite[Ch.~4]{blowup:algebras}, and \cite[\S 2.4]{Lefschetz:properties} for further details.

Given an ideal $I \subseteq S$ with $\hgt I = c$, we say that $I$ is \emph{directly linked} to the ideal $J$ if there is a complete intersection $L \subseteq I$ with $\hgt L = c$ such that $J = (L : I)$ and $I = (L : J)$.  When $I$ is unmixed, in particular when $I$ is Cohen-Macaulay, it is automatically directly linked to the colon ideal $J = (L : I)$ for any complete intersection $L$ as above.  Many properties can be passed between an ideal and its links.  For example, if $I$ and $J$ are directly linked, then $S/I$ is Cohen-Macaulay if and only if $S/J$ is.  We will especially need the following proposition relating the $h$-vectors of a Cohen-Macaulay ideal and its links, which was originally proved in greater generality in \cite[3(b)]{Gorenstein:algebras:and:Cayley-Bacharach}; see also \cite[5.2.19]{Migliore:linkage:book}.

\begin{prop} \label{h-vectors:and:linkage}
Suppose that $I \subseteq S$ is a Cohen-Macaulay ideal generated by quadrics and that $L \subseteq I$ is a quadratic complete intersection with $\hgt L = \hgt I = c$.  Denote the linked ideal by $J = (L: I)$.  Then $h$-vectors of $S/I$ and $S/J$ are related by
\[ h_i(S/I) = \binom{c}{i} - h_{c-i}(S/J) \]
\end{prop}

%\begin{proof}
%Since the conclusion is preserved under flat base change, we may assume the ground field is infinite.  By linkage, we have an isomorphism 
%\[ \omega_{S/J} \iso \Hom_S(S/J, S/L(a)) \iso (0 :_{S/L} J/L)(a) = I/L(a) \] 
%where $a = a(S/L)$ is the $a$-invariant of $S/L$.  The additivity of the Hilbert series along the short exact sequence $0 \to \omega_{S/J}(-a) \to S/L \to S/I \to 0$ implies an equality of $h$-vectors $h_i(S/I) = h_i(S/L) - h_{i-a}(\omega_{S/J}) = \binom{c}{i} - h_{i-a}(\omega_{S/J})$.  
%
%If $x$ is a linear form which is regular on $S/J$, then $x$ is also regular on the maximal Cohen-Macaulay module $\omega_{S/J}$ so that 
%\begin{align*} 
%\omega_{S/J}/x\omega_{S/J} &\iso \Ext^{c+1}_S(S/(J, x), S)(-n) \\
%& \iso \Ext^c_{S/xS}(S/(J, x), S/xS)(-n) \iso \omega_{S/(J, x)}(-1)
%\end{align*}
%where $n = \dim S$.  Hence, if $\seq{x} = x_1, \dots, x_{n-c}$ denotes a maximal regular sequence of linear forms on $S/J$ and we set $S' = S/\seq{x}S$ and $J' = (J, \seq{x})/\seq{x}S$, then $\omega_{S/J}/\seq{x}\omega_{S/J} \iso \omega_{S'/J'}(-n+c)$.  Since the $h$-vector is preserved by killing a regular sequence of linear forms, we have $h_{i-a}(\omega_{S/J}) = h_{i-a}(\omega_{S'/J'}(-n+c)) = h_{i-a-n+c}(\omega_{S'/J'})$.  Since $S'/J'$ is Artinian, the $h$-vector coincides with the Hilbert function and $\omega_{S'/J'}$ is isomorphic to the graded $\kk$-vector space dual ${}^*\Hom_\kk(S'/J', \kk)$ so that $h_{i-a}(\omega_{S/J}) = h_{a+n-c-i}(S'/J') = h_{c-i}(S/J)$ since $a + n = \reg(S/L) + c = 2c$.
%\end{proof}

Suppose that $S = \kk[x_0, \dots, x_{c-1}]$ so that $R = S/I$ is an Artinian ring.  We consider the ring $D = \kk[y_0,\dots, y_{c-1}]$ as an $S$-module via the action defined as follows.  For any monomials $x^\alpha = x_0^{\alpha_0} \cdots x_{c-1}^{\alpha_{c-1}}$ and $y^\beta = y_0^{\beta_0} \cdots y_{c-1}^{\beta_{c-1}}$ with $\alpha, \beta \in \NN^c$, we define
\[ x^\alpha \cdot y^\beta = \left\{\begin{array}{cl} y^{\beta - \alpha} & \text{if $\beta_i \geq \alpha_i$ for all $i$} \\ 0 & \text{otherwise} \end{array}\right. \]
The $S$-module action on $D$ is defined by extending this action on
monomials linearly.  We call $D$ the module of \emph{inverse  polynomials} since $D$ is isomorphic as an $S$-module to 
\[H^c_{S_+}(S) = \frac{1}{x_0\cdots x_{c-1}}\kk[x_0^{-1}, \dots,  x_{c-1}^{-1}]. \] 
The set $M = (0:_D I)$ is a finitely generated $S$-submodule of $D$ called the \emph{Macaulay inverse system} of $I$, and it is well-known that there is bijective correspondence between ideals $I \subseteq S$ such that $R = S/I$ is Artinian and finitely generated $S$-submodules of $D$.  Given such a submodule $M \subseteq D$, we associate to it the ideal $I = (0:_S M)$.  There is a close relationship between the Hilbert functions of $I$ and $M$.  In particular, $R = S/I$ is Gorenstein with socle in degree $r$ if and only if $M$ is generated by a single inverse polynomial of degree $r$.

Over a field of characteristic zero, it is possible to describe inverse systems by an equivalent action of $S$ on $D$ using partial derivatives.  In the following sections, we will use only the contraction action defined in the preceding paragraph.

In the next section, we will consider quadratic Gorenstein rings $R$ with $c = r + 1$.  We close this background section with a proposition that explains why this is the first interesting case in which to ask whether or not $R$ is Koszul.

\begin{prop}[{\cite[3.1]{minimal:homogeneous:linkage}}] \label{regularity:bound}
Suppose that $R = S/I$ is a quadratic Cohen-Macaulay ring. Then $\reg R \leq \pd_S R$, and equality holds if and only if $R$ is a complete intersection.
\end{prop}

\section{Quadratic Gorenstein Rings with $c = r+1$} \label{deviation:2:case}

In this section, we show that every quadratic Gorenstein ring with $c = r + 1$ is Koszul.  Using linkage, Migliore and Nagel have already computed the $h$-vectors of such rings \cite[3.1]{Gorenstein:algebras:presented:by:quadrics}.  By carefully analyzing their arguments, we are actually able to prove much more.

\begin{thm}
Let $R = S/I$ be a quadratic Cohen-Macaulay ring with $c = \hgt I$.  Then $R$ is a Koszul almost complete intersection if and only if its $h$-vector is given by
\[ h_i(R) = \binom{c}{i} - \binom{c-2}{i-2} \]
\end{thm}

\begin{proof}
If $R$ is a Koszul almost complete intersection, it follows from \cite[3.3]{Koszul:ACI's} that the $h$-polynomial of $R$ is $h_R(t) = (1+t)^{c-2}(1+ 2t)$ so that
\begin{align*} 
h_i(R) &= 2\binom{c-2}{i-1} + \binom{c-2}{i} = \binom{c-2}{i-1} + \binom{c-1}{i} \\
&= \binom{c-2}{i-1} + \binom{c-2}{i-2} - \binom{c-2}{i-2} + \binom{c-1}{i} = \binom{c}{i} - \binom{c-2}{i-2}
\end{align*}

Conversely, if $R$ is a quadratic Cohen-Macaulay ring with the given $h$-vector, then the $h$-polynomial of $R$ is $h_R(t) = (1+t)^c - t^2(1+t)^{c-2} = (1+t)^{c-2}(1+2t)$ so that
\begin{align} \label{ACI:h-polynomial}
\begin{split}
\sum_j (\sum_{i = 0}^j (-1)^i \beta_{i,j}^S(R))t^j &= (1-t)^ch_R(t) = (1-t^2)^{c-2}(1-t)^2(1+2t) \\
&\equiv (1 - (c-2)t^2 + \binom{c-2}{2}t^4)(1 - 3t^2 + 2t^3) \pmod {t^5} \\
&\equiv 1 -(c+1)t^2 + 2t^3 + (3(c-2) + \binom{c-2}{2})t^4 \pmod {t^5}
\end{split}
\end{align}  
Since $I$ is generated by quadrics, this implies that $\beta_{1,2}^S(R) = c + 1$ and $\beta_{2,3}^S(R) = 2$.  It then follows from \cite[4.2]{Koszul:algebras:defined:by:3:quadrics} that $\beta_{3, 4}^S(R) = 0$ so that $\beta_{2, 4}^S(R) = 3(c-2) + \binom{c-2}{2}$.

Since the conclusion is preserved under flat base change, we may assume that the ground field is infinite.  It follows from \cite[4.2]{Koszul:ACI's} that, if $R$ is a quadratic almost complete intersection over an infinite ground field with $\beta_{2,3}^S(R) \geq 2$, we can write $I = (q_1, \dots, q_{c+1})$ for some quadrics $q_i$, where $q_2, \dots, q_{c+1}$ form a regular sequence and $q_1, q_2, q_3$ are the corresponding minors of a $3 \times 2$ matrix $M = (\ell_{i,j})$ of linear forms.  In particular, we have $\hgt I_2(M) = 2$ as $(q_2, q_3) \subseteq I_2(M)$.  Additionally, \[(q_2, \dots, q_{c+1}) \subseteq J' = (\ell_{1,1}, \ell_{1,2}, q_4, \dots, q_{c+1})\] so that $\hgt J' = c$ and $\ell_{1,1}, \ell_{1,2}, q_4, \dots, q_{c+1}$ is a regular sequence.  Now, set $L = (q_2, \dots, q_{c+1})$, and consider the linked ideal $J = (L : I)$.  Since $q_2$ and $q_3$ are minors of $M$, we have $\ell_{1,1}q_1, \ell_{1,2}q_1 \in (q_2, q_3)$ so that $J' \subseteq J$.  Then $J$ and $J'$ are ideals of height $c$ with the same $h$-vector by Proposition \ref{h-vectors:and:linkage}.  Hence, they have the same Hilbert function and are equal.

From the natural short exact sequence $0 \to S/J(-2) \to S/L \to R \to 0$, we can obtain a free resolution of $R$ from the free resolutions of $S/J$ and $S/L$ via a mapping cone.  In particular, this yields 
\[\beta^S_{2,j}(R) \leq \beta_{2,j}^S(S/L) + \beta^S_{1,j-2}(S/J)\]  
for all $j$.  Combining this bound with the description of $J$ from the previous paragraph, we see that $\beta_{2,j}^S(R) = 0$ for $j \geq 5$.  

This shows that $\syz_1^S(I)$ is generated by linear syzygies and Koszul syzygies on the minimal generators of $I$.  In fact, the columns of $M$ must be the two independent linear syzygies of $I$, and the $S$-span of these linear syzygies contains the Koszul syzygies involving any two of $q_1, q_2, q_3$.  Let $W$ denote the $\kk$-span of the Koszul syzygies involving some $q_j$ for $j \neq 1, 2, 3$.  We claim that $W \cap S_+(\syz_1^S(I)_3) = 0$.  If not, there is some nonzero $\kk$-linear combination of Koszul syzygies 
\[ \sum_{\substack{1 \leq i < j \leq c+1 \\ j \geq 4}} a_{i,j}(q_ie_j - q_je_i) \]
(where $e_i$ denotes the $i$-th standard basis vector of $S(-2)^{c+1}$) which is an $S$-linear combination of the linear syzygies.  But such an $S$-linear combination must be zero in its $j$-th coordinate for all $j \geq 4$.  If $a_{i,j} \neq 0$, then examining the $j$-th coordinate of the above Koszul syzygy yields a linear dependence relation on the $q_i$, which is impossible since the $q_i$ are the minimal generators for $I$.  Hence, the claim holds so that the Koszul syzygies in $W$ are part of a minimal set of generators for $\syz_1^S(I)$.  As $\beta_{2,j}^S(R) = 0$ for $j \geq 5$ and $\beta_{2,4}^S(R) = 3(c-2) + \binom{c-2}{2}$ is precisely the number of Koszul syzygies in $W$, we see that $\syz_1^S(I)$ is minimally generated by the columns of $M$ and the Koszul syzygies in $W$.

Consequently, we see that $q_{c+1}$ is a nonzerodivisor modulo $(q_1, \dots, q_c)$.  Then $S/(q_1, \dots, q_c)$ is Cohen-Macaulay and has $h$-polynomial given by $(1+t)h_{S/(q_1, \dots, q_c)}(t) = h_R(t)$, so that $\deg h_{S/(q_1, \dots, q_c)}(t) = c-2$ and
\[ h_i(S/(q_1, \dots, q_c)) = h_{i+1}(R) - h_{i+1}(S/(q_1, \dots, q_c)) = \binom{c-1}{i} - \binom{c-1}{i-2}   \]
by a downward induction.  Hence, an induction on $c$ shows that $q_4, \dots, q_c$ is a regular sequence modulo $I_2(M)$, and it follows from \cite[3.3]{Koszul:ACI's} that $R$ is Koszul.
\end{proof}

\begin{thm} \label{c=r+1:is:Koszul}
Let $R = S/I$ be a quadratic Gorenstein ring over an infinite ground field with $\reg R +1 = \hgt I = c$.  Then $I = \Pf(M) + (q_6, \dots, q_{c+2})$ for some $5 \times 5$ alternating matrix $M$ of linear forms such that $\hgt \Pf(M) = 3$ and some quadrics $q_6, \dots, q_{c+2}$ which are a regular sequence modulo $\Pf(M)$.  Moreover, $R$ is Koszul.
\end{thm}

\begin{proof}
Since $\reg R < \hgt I = c$, $I$ is not complete intersection, and because $R$ is a quadratic Gorenstein ring, this implies that $\reg R \geq 2$ so that $c \geq 3$.  By \cite[3.1]{Gorenstein:algebras:presented:by:quadrics}, the $h$-vector of $R$ is given by
\[ h_i(R) = \binom{c-1}{i} + \binom{c-3}{i-1} \]
so that the $h$-polynomial is $h_R(t) = (1+t)^{c-1} + t(1+t)^{c-3} = (1+t)^{c-3}(1+3t+t^2)$.  As in the proof of the previous theorem, this implies that 
\begin{align} \label{Gorenstein:h-polynomial}
\begin{split}
\sum_j (\sum_{i = 0}^j (-1)^i \beta_{i,j}^S(R))t^j &= (1-t)^ch_R(t) = (1-t^2)^{c-3}(1-t)^3(1+3t+t^2) \\
&\equiv (1 - (c-3)t^2 + \binom{c-3}{2}t^4)(1 - 5t^2 + 5t^3) \pmod {t^5} \\
&\equiv 1 -(c+2)t^2 + 5t^3 + (5(c-3) + \binom{c-3}{2})t^4 \pmod {t^5}
\end{split}
\end{align}  
Since $I$ is generated by quadrics, this implies that $\beta_{1,2}^S(R) = c + 2$ and $\beta_{2,3}^S(R) = 5$.  We can then write $I = (q_1, \dots, q_{c+2})$ for some quadrics $q_i$ with $q_1, \dots, q_c$ a regular sequence.  Set $L = (q_1, \dots, q_c)$, and consider the linked ideal $J = (L : I)$.  Then Proposition \ref{h-vectors:and:linkage} shows that the $h$-vector of $S/J$ is given by
\[ h_i(S/J) = \binom{c-1}{i} - \binom{c-3}{i-2} \]
This implies that $J$ contains a linear form $z$, and $J/zS$ is a Koszul almost complete intersection by the preceding theorem.  

Next, we claim that $J = (L, z)$.  Since $J/zS$ is minimally generated by $c$ quadrics, it is enough to show that $q_1, \dots, q_c$ are independent modulo $z$.  Suppose this is not the case.  We note that $L \subseteq (L, z) \subseteq J$ implies $\hgt (L, z) = c$.  Hence, we may assume that $q_1, \dots, q_{c-1}$ are a regular sequence modulo $z$ and $q_c \in (q_1, \dots, q_{c-1}, z)$, and so, after a suitable change of generators for $L$, we may further assume that $q_c = z\ell$ for some linear form $\ell$.  In addition, we know that $zq_{c+1} = \ell_1q_1 + \cdots + \ell_cq_c$ for some linear forms $\ell_i$ as $zq_{c+1} \in L$.  Modulo $z$, this corresponds to a linear syzygy on the images of $q_1, \dots, q_{c-1}$, but since $q_1, \dots, q_{c-1}$ are a regular sequence modulo $z$, this syzygy must be trivial.  Hence, we have $\ell_i = \lambda_iz$ for some $\lambda_i \in \kk$, and after cancelling $z$, we have $q_{c+1} = \lambda_1q_1 + \cdots + \lambda_{c-1}q_{c-1} + \ell_c\ell$.  And so, after another change of generators for $I$, we may assume that $q_{c+1} = \ell_c\ell$ and, similarly, that $q_{c+2} = h_c\ell$ for some linear form $h_c$.  In this case, the linear syzygies of $I$ contain the syzygies on $q_c, q_{c+1}, q_{c+2}$ corresponding to the columns of the matrix
\[ \begin{pmatrix} \ell_c & h_c & 0 \\ -\ell & 0 & h_c \\ 0 & -\ell & -\ell_c \end{pmatrix} \]
which in turn has a linear second syzygy $(h_c, -\ell_c, \ell)$, so that $\beta_{3,4}^S(R) \neq 0$.  Since $R$ is Gorenstein with $\reg R = c - 1$, this implies that $\beta_{c-3, 2c-5}^S(R) \neq 0$.  The short exact sequence $0 \to I/L \to S/L \to R \to 0$ induces an
exact sequence
 \[0 = \Tor_{c-3}^S(S/L, \kk)_{2c-5} \to \Tor_{c-3}^S(R, \kk)_{2c-5} \to \Tor_{c-4}^S(I/L, \kk)_{2c-5}\] 
 so that $\beta^S_{c-4,2c-5}(I/L) \neq 0$.  However, we know that 
\[I/L \iso \Hom_S(S/J, S/L) \iso \omega_{S/J}(2c-n) \iso \Ext^c_S(S/J, S)(2c) \]  
where $n = \dim S$.  Since $S/J$ is Cohen-Macaulay, it follows that \[\beta_{4,5}^S(S/J) = \beta^S_{c-4,-5}(\Ext^c_S(S/J, S)) = \beta^S_{c-4,2c-5}(I/L) \neq 0.\]  However, as $J/zS$ is a Koszul almost complete intersection, \cite[3.3]{Koszul:ACI's} implies that $\beta_{3,4}^{S/zS}(S/J) = 0$ so that $\beta_{4,5}^S(S/J) = 0$, which is the desired contradiction.  Therefore, we see that $J = (L, z)$ as claimed.

After a suitable change of generators for $L$ modulo $z$ there is a $3 \times 2$ matrix $\bar{M} = (\bar{\ell}_{i,j})$ of linear forms in $\bar{S} = S/zS$ such that the images of $q_1, q_2, q_3$ are the corresponding minors of $\bar{M}$.  Let $M = (\ell_{i,j})$ be a lift of this matrix of linear forms to $S$.  Then we have
\begin{align*}
q_1 &= - zu_1 + \ell_{2,1}\ell_{3,2} - \ell_{2,2}\ell_{3,1}  \\
q_2 &= zu_2 -\ell_{1,1}\ell_{3,2} + \ell_{1,2}\ell_{3,1} \\
q_3 &= - zu_3  + \ell_{1,1}\ell_{2,2} - \ell_{1,2}\ell_{2,1}
\end{align*}
for some linear forms $u_i$.  That is, we can obtain $q_1, q_2, q_3$ as submaximal Pfaffians of the alternating matrix
\[ M = \begin{pmatrix}
0 & u_3 & u_2 & \ell_{1, 1} & \ell_{1,2} \\
-u_3 & 0 & u_1 & \ell_{2, 1} & \ell_{2,2} \\
-u_2 & -u_1 & 0 & \ell_{3,1} & \ell_{3,2} \\
-\ell_{1,1} & -\ell_{2,1} & -\ell_{3,1} & 0 & z \\
-\ell_{1,2} & -\ell_{2,2} & -\ell_{3,2} & -z & 0
\end{pmatrix}
\]
The other two pfaffians of $M$ are
\begin{align*}
f_4 &=\ell_{3,2}u_3 - \ell_{2,2}u_2 + \ell_{1,2}u_1  \\
f_5 &= -\ell_{3,1}u_3 + \ell_{2,1}u_2 - \ell_{1,1}u_1
\end{align*}
Since $q_1, q_2, q_3$ is a regular sequence, $\hgt \Pf(M) = 3$ so that $\Pf(M)$ is a Gorenstein ideal with minimal free resolution given by the Buchsbaum-Eisenbud structure theorem \cite[3.4.1]{Bruns:Herzog}.  In particular, $M$ is the matrix of first syzygies on $\Pf(M)$ so that $zf_4, zf_5 \in (q_1, q_2, q_3)$, and hence, we have $f_4, f_5 \in (L : z) = (L : J) = I$.  Additionally, $q_1, q_2, q_3, f_4, f_5$ must be independent quadrics since otherwise the Buchsbaum-Eisenbud complex would not be a resolution of $\Pf(M)$.  Hence, after replacing the original $q_i$ for $i \geq 4$, we may assume that $I = \Pf(M) + (q_6, \dots, q_{c+2})$ for some quadrics $q_i$.  

As the columns of $M$ are independent linear syzygies on $I$ and $\beta_{2,3}^S(R) = 5$, these must be all of the linear syzygies of $I$.  Consequently, we must have $\beta_{3,4}^S(R) = 0$ as there are no linear syzygies on the columns of $M$, and so, \eqref{Gorenstein:h-polynomial} implies that \[\beta_{2,4}^S(R) = 5(c-3) + \binom{c-3}{2}.\]  This is precisely the number of Koszul syzygies involving some $q_j$ for $j \geq 6$; the other Koszul syzygies are in the $S$-span of the linear syzygies.  By arguing inductively as in the proof of the preceding theorem, we see that $q_6, \dots, q_{c+2}$ must be a regular sequence modulo $\Pf(M)$.
\end{proof}

By flat base change to an infinite ground field, we have the following corollary. 

\begin{cor}
Let $R = S/I$ be a quadratic Gorenstein ring with $\reg R +1 = \hgt I = c$.  Then $R$ is Koszul with multiplicity $e(R) = 5 \cdot 2^{c-3}$ and Betti table 
\begin{center}
\scalebox{0.9}{
\textnormal{
\begin{tabular}{c|ccccccccc} 
  & 0 & 1 & 2 & 3 & 4 & $\cdots$ & $c-2$ & $c-1$ & $c$ \\ 
\hline 
0 & 1 & -- & -- & -- & -- & & -- & -- & -- \\ 
1 & -- & $c+2$ & 5 & -- & -- & & -- & -- & --  \\
2 & -- & -- & $5\binom{c-3}{1}+\binom{c-3}{2}$ & $5\binom{c-3}{1} + 1$ & -- & & -- & -- & --  \\
3 & -- & -- & -- & $5\binom{c-3}{2} + \binom{c-3}{3}$ & $5\binom{c-3}{2} + \binom{c-3}{1}$ \\
$\vdots$ & & & & & $\ddots$ & $\ddots$ \\
$c-2$ & -- & -- & -- & -- & & & $5$ & $c + 2$ & -- \\
$c-1$ & -- & -- & -- & -- & & & -- & -- & 1
\end{tabular}
}}
\end{center}
Specifically, we have 
\[
\beta_{i,2i}^S(R) = 5\binom{c-3}{i-1}+ \binom{c-3}{i}
\quad \text{and} \quad 
\beta_{i,2i-1}^S(R) = 5\binom{c-3}{i-2} + \binom{c-3}{i-3}
\] 
so that 
\[\beta_i^S(R) = \binom{c}{i} + 2\binom{c-2}{i-1} \]
for all $i$.
\end{cor}

As a consequence of the above corollary, we see that $\binom{c}{i} \leq \beta_i^S(R) \leq \binom{c+2}{i}$ for all $i$ so that both the Betti number bounds proposed by Buchsbaum-Eisenbud-Horrocks Conjecture and by \cite[6.5]{free:resolutions:over:Koszul:algebras} hold for this class of Koszul algebras.

\begin{question}
Using the fact that we know the entire Betti table of $R$ and that $R$ is Koszul from the preceding corollary, does the conclusion of Theorem \ref{c=r+1:is:Koszul} hold without the infinite ground field hypothesis?
\end{question}

We suspect that the answer to the above question is yes since such a structure theorem is possible for Koszul almost complete intersections without any restriction on the ground field.

\begin{question}
Is every quadratic Gorenstein ring $R = S/I$ where $I$ has $c + 2$ minimal quadric generators necessarily of the form described in Theorem \ref{c=r+1:is:Koszul}?
\end{question}

\noindent Huneke-Ulrich ideals \cite{Huneke-Ulrich:ideals} are a well-known class of Gorenstein ideals of deviation two, but these ideals are never quadratic in codimension greater than three.  It would also be interesting to explore when such rings are G-quadratic or admit a Gr\"obner flag.

\section{Quadratic Gorenstein Rings with $r = 4$ Which Are Not Koszul}
\label{r>=4}

\begin{thm} \label{r=4:c>=7} 
Let $S = \kk[x_0,\ldots,x_{c-1}]$ where $c \geq 7$, and let $D = \kk[y_0, \dots, y_{c-1}]$ denote the module of inverse polynomials as in \S \ref{background}.  Consider the inverse polynomial
\[
F = \sum_{i \in \ZZ/c\ZZ}  y_iy_{i+1}y_{i+2}^2
\]
and let $I_F = (0:_S F) \subseteq S$ be the corresponding ideal.  Then $R = S/I_F$ is a quadratic Gorenstein ring which is not Koszul, with $h$-vector $h(R) = (1, c, 2c, c,1)$.
\end{thm}

In the proof below, we view $\ZZ/c\ZZ$ as being totally ordered by $0 < 1 < \cdots < c-1$ and set $\abs{i} = \min\{i, c - i\}$ for each $i \in \ZZ/c\ZZ$.

\begin{proof} 
It is well-known that $R$ is Gorenstein with regularity 4.  For each
$i \in \ZZ/c\ZZ$, we note that \[x_iF = y_{i+1}y_{i+2}^2 +
  y_{i-1}y_{i+1}^2 + y_{i-2}y_{i-1}y_i.\]  In particular, all of these
cubics are linearly independent since, for example, only $x_iF$
contains the monomial $y_{i+1}y_{i+2}^2$, and so, no linear form
annihilates $F$ so that $I_F$ is nondegenerate.  Hence, the $h$-vector
of $R$ has the form \[h(R) = (1, c, h_2, c, 1), \]
where $h_2$ is $\binom{c+1}{2}$ minus the number of linearly independent quadrics in $I_F$.  Any quadratic monomial $x_ix_j$ with $\abs{i - j} \geq 3$ annihilates $F$, and a count shows there are $\frac{c(c-5)}{2}$ such monomials. The $c$ binomials of the form \[x_i^2-x_{i-1}x_{i-3}\] also annihilate $F$; combining these with the preceding square-free monomials gives a total of $\frac{c(c-3)}{2}$ linearly independent quadrics that annihilate $F$. Thus, to show that \[h(R) = (1,c,2c,c,1),\] it suffices to prove that there are no other quadrics in $I_F$. 

Let $J \subseteq I_F$ be the ideal generated by the $\frac{c(c-3)}{2}$ quadrics of the preceding paragraph.  We will show that $J = I_F$ so that $R$ is quadratic with the desired $h$-vector.  If $q \in I_F$ is a quadric, then after replacing $q$ with suitable linear combination with quadrics in $J$ we assume that $q$ has the form
\[
q = \sum_i a_{i, i+1}x_ix_{i+1} + \sum_i a_{i, i+2}x_ix_{i+2}
\] 
for some $a_{i,i+1}, a_{i,i+2} \in \kk$.  We claim that such a quadric $q$ is zero.  Indeed, since
\begin{align*}
x_ix_{i+1}F &= y_{i+2}^2+ y_{i-1}y_{i+1}  \\
x_ix_{i+2}F &= y_{i+1}y_{i+2}
\end{align*}
we see that $x_ix_{i+1}F$ is the only quadric containing $y_{i+2}^2$, and this forces $a_{i,i+1} = 0$ for all $i$ as $qF = 0$.   But then \[0 = qF = \sum a_{i,i+2}y_{i+1}y_{i+2}\] is a sum of independent quadratic monomials so that $a_{i,i+2} = 0$ for all $i$ as well.  Hence, $q = 0$ as claimed, and we see that $[I/J]_2 = 0$.

Among the degree 3 square-free monomials $x_ix_jx_\ell$ with $i < j < \ell$, we see that $x_ix_jx_\ell$ is zero modulo $J$ unless $\abs{i - \ell} \leq 2$ so that only the monomials $x_ix_{i+1}x_{i+2}$ could possibly be nonzero.  Similarly, any monomial of the form $x_i^2x_j$ is zero modulo $J$ unless $\abs{i - j} \leq 2$.  If $j = i, i +1, i+2$, then 
 \begin{equation} \label{degree:3:relations}
 x_i^2x_j \equiv x_{i-3}x_{i-1}x_j \equiv 0 \pmod{J}
 \end{equation}
 since 
\[x_{i-3}x_i, x_{i-3}x_{i+1}, x_{i-1}x_{i+2} \in J \mbox{ as }c \geq 7.\]  On the other hand, we have 
 \begin{align*} 
 x_i^2x_{i-1} &\equiv x_{i-3}x_{i-1}^2 \equiv x_{i-4}x_{i-3}x_{i-2} \\
 x_i^2x_{i-2} &\equiv x_{i-3}x_{i-2}x_{i-1}
 \end{align*}
 modulo $J$ so that the $c$ monomials of the form $x_ix_{i+1}x_{i+2}$ span $S/J$ in degree 3.  In degree 4, we note that \[x_jx_ix_{i+1}x_{i+2} \equiv 0 \pmod{J}\] unless $j = i, i + 1, i+2$, and moreover, we have 
\begin{align*}
x_i^2x_{i+1}x_{i+2} &\equiv x_{i-3}x_{i-1}x_{i+1}x_{i+2} \equiv 0 \\
x_ix_{i+1}^2x_{i+2} &\equiv x_{i-2}x_i^2x_{i+2} \equiv 0  \\
x_ix_{i+1}x_{i+2}^2 &\equiv x_{i-1}x_ix_{i+1}^2
\end{align*}
for all $i$, where the middle equivalence follows from \eqref{degree:3:relations}.  Thus, $S/J$ is spanned in degree 4 by the monomials $x_ix_{i+1}x_{i+2}^2$, which are all equivalent to one another.  In particular, every variable annihilates $x_0x_1x_2^2 \equiv x_3x_4x_5^2$ modulo $J$ so that every monomial of degree at least five is zero modulo $J$.  Since $S/J$ maps surjectively onto $R$ and $h(R) = (1, c, 2c, c, 1)$, this shows that $J = I_F$.

To prove that $R$ is not Koszul, it suffices to show that there is a quadratic first syzygy of $I$ which is not in the module $Z$ generated by the linear syzygies and the Koszul syzygies of $I$ (see for example \cite[2.8]{Koszul:ACI's}). For each $i \in \ZZ/c\ZZ$, let \[q_i = x_i^2 - x_{i-1} x_{i-3}\] be the binomial quadrics in $I_F$, 
and let $u_i =x_{i-1} x_{i-2}$. Then a computation shows that 
\[
\sum_{i=0}^{c-1} q_i u_i = 0.
\]
Call this syzygy $u = (u_0, \dots, u_{c-1})$; we claim that $u \notin Z$. 

To see this, let $m_1, \dots, m_s$ denote the quadratic monomials in $I_F$, and let \[L = (q_1, \dots, q_{c-1}, m_1, \dots, m_s).\]  We further claim that \[(L : q_0)_1 = (x_2, \ldots, x_{c-3})_1. \] Indeed, for $3 \leq i \leq c -4$, we have $x_iq_0 \in (m_1, \dots, m_s)$, while
\begin{align*}  
x_2q_0 &= -x_0q_3 + x_3(x_0x_3) - x_{c-3}(x_2x_{c-1}) \\
x_{c-3}q_0 &= x_0(x_0x_{c-3}) -x_{c-1}q_{c-3} + x_{c-6}(x_{c-1}x_{c-4})
\end{align*} 
are both elements of $L$. If $\ell \in (L : q_0)_1$ is any other
linear form, we may assume that \[\ell = a x_0 + b x_1 + d x_{c-2} + e  x_{c-1} \mbox{ for some }a, b, d, e \in \kk.\]  Since none of the generators of $L$ contain a monomial of the form $x^2_0$, $x_{c-1}x_{c-3}$, or $x_ix_{i+1}$ in their supports, we see that:
\begin{itemize}
\item $a = 0$ as $x_0^3$ is not in the support of any polynomial in $L$.  
\item $b = 0$ as $x_1 x_0^2$ is not in the support of any polynomial in $L$.
\item $d = 0$ as $x_{c-1} x_{c-2} x_{c-3}$ is not in the support of any polynomial in $L$.
\item $e = 0$ as $x_{c-1} x_0^2$ is not in the support of any polynomial in $L$.
\end{itemize}
Hence, $\ell = 0$ so that $(L : q_0)_1 = (x_2, \dots, x_{c-1})_1$ as claimed.

As a consequence of the preceding paragraph, we see that the first coordinate of any linear syzygy on $q_0, \dots, q_{c-1}, m_1, \dots, m_s$ must belong to $(x_2, \dots, x_{c-1})$.  If $u \in Z$, then its first coordinate $u_0 = x_{c-1}x_{c-2}$ must be a linear combination of the first coordinates of the linear syzygies and Koszul syzygies of $I_F$ so that \[u_0 \in (x_2, \dots, x_{c-3}, q_1, \dots, q_{c-1}) = (x_2, \dots, x_{c-3}, x_1x_{c-1}, x_{c-2}^2, x_{c-1}^2, q_1). \] But this is impossible since $x_{c-1}x_{c-2}$ does not appear in any quadric in this ideal.  This shows that $u$ is a minimal quadratic syzygy which is not in the submodule generated by the linear and Koszul syzygies. Therefore $S/I_F$ is quadratic, Gorenstein, and not Koszul, with regularity four and codimension $c$.
\end{proof}

\begin{example} \label{generalizing:Matsuda}
The $c = 7$ case of Theorem~\ref{r=4:c>=7} yields an ideal with
generators
\[
\begin{array}{ccc}
y_0y_3 & &y_3y_5-y_6^2\\ 
  y_0y_4& &y_2y_4-y_5^2\\
y_1y_4& & y_1y_3-y_4^2\\
y_1y_5& &y_0y_2-y_3^2 \\ 
y_2y_6& &y_1y_6-y_2^2\\
y_2y_5& &y_0y_5-y_1^2\\
y_3y_6& &y_4y_6 -y_0^2
\end{array}
\] 
This recovers the Artinian reduction of the toric ring in Example \ref{Matsuda's:example} (see the proof of \cite[1.3]{Matsuda}); in this sense, the above result greatly extends Matsuda's example.
\end{example}

\begin{example}\label{r=4:c=6}
When $\kk = \QQ$, we can also find an example of a non-Koszul quadratic Gorenstein ring with regularity $r = 4$ and codimension $c = 6$ by slightly modifying the construction of the preceding theorem.  As in Theorem~\ref{r=4:c>=7}, we start
with the polynomial 
\[
F = y_0y_1y_2^2+y_1y_2y_3^2+y_2y_3y_4^2+y_3y_4y_5^2+y_4y_5y_0^2+y_5y_0y_1^2
\]
The corresponding ideal $I_F$ is generated by six quadrics and two
cubics. It is possible to eliminate the cubics by modifying the input
polynomial $F$ to
\[
G=F+y_0y_5y_4^2+y_0y_5^3
\]
The ideal $I_G$ is generated by the nine quadrics
\[
\begin{array}{c}
y_2y_5\\
y_1y_4\\
y_0y_3 \\
y_3^2-y_0y_2 \\ 
y_0^2-y_3y_5    \\
y_2^2-y_1y_5\\ 
y_1^2+y_2y_4-y_5^2  \\ 
y_1y_3-y_2y_4-y_4^2+y_5^2 \\ 
y_0y_4+y_2y_4-y_3y_5-y_5^2
\end{array}
\]
The ring $R = S/I_G$ has Betti table
\vspace{1 ex}
\begin{equation}
\scalebox{0.9}{
\begin{tabular}{r|ccccccc}
 & 0 & 1 & 2 & 3 & 4 & 5 & 6   \\
\hline  
0 & 1 & -- & -- & -- & -- & -- & --  
\\ 
1 & -- & 9 & 4 & -- & -- & -- & -- 
\\ 
2 & -- & -- & 40 & 72 & 40 & -- & -- 
\\ 
3 & -- & -- & -- & -- & 4 & 9 & --  
\\ 
4 & -- & -- & -- & -- & -- & -- & 1  
\end{tabular}
}
\vspace{1 ex}
\end{equation} 
The $h$-vector of $R$ is $(1, 6, 12, 6, 1)$, and since $\beta_{2,4}^S(R) = 40 > \binom{9}{2}$, $R$ cannot be Koszul by \cite[3.4]{Koszul:algebras:defined:by:3:quadrics}.
\end{example}

For any non-Koszul quadratic Gorenstein ring $R = S/I$, let $R' = R\otimes_\kk \kk[z]/(z^2)$, where $z$ is a new variable. Then $R'$ is also non-Koszul, quadratic and Gorenstein, and has a mapping cone resolution over the polynomial ring $S[z]$ with 
\[
\reg R' = \reg R +1 \mbox{ and } \codim R' = \codim R + 1 
\]
(see for example \cite[4.9]{QGRKP1} for further details).  Applying this to the rings appearing in Theorem \ref{r=4:c>=7} yields non-Koszul quadratic Gorenstein rings when $r=5$ and $c \in \{9, 10\}$, and applying it to Example \ref{r=4:c=6} yields non-Koszul quadratic Gorenstein rings when $c=r+2$. Combined with Theorem \ref{c=r+1:is:Koszul}, we obtain the improvement from Figure \ref{table:of:when:quadratic:Gorenstein:rings:are:Koszul} to Figure \ref{improved:table:of:when:quadratic:Gorenstein:rings:are:Koszul} for quadratic Gorenstein rings over a field of characteristic zero.  

In recent work \cite{jason}, McCullough-Seceleanu use the idealization construction to produce a quadratic Gorenstein algebra with $r=3$ and $c=8$ which is not Koszul. 
We are working to apply the techniques of Caviglia \cite{Caviglia} to understand the two remaining unknown cases $r=3$ and $c \in \{6,7\}$ in Figure \ref{improved:table:of:when:quadratic:Gorenstein:rings:are:Koszul}. 

\begin{rmk}
Although our strongest results hold only in characteristic zero, most of our results in this paper and \cite{QGRKP1} hold in arbitrary characteristic so that Figure \ref{improved:table:of:when:quadratic:Gorenstein:rings:are:Koszul} remains mostly intact in prime characteristic.  The notable exceptions are that the cases where $6 \leq c = r + 2$ and where $r = 3$ and \[c \in \{10, 11, 14, 15, 18, 19, 24\}\] remain unknown.
\end{rmk}

\begin{ack}
{\tt Macaulay2} computations were essential to our work. The first author thanks Paolo Mantero for informing him of the references for Propositions \ref{h-vectors:and:linkage} and \ref{regularity:bound}.  We thank BIRS-CMO, where we learned of Matsuda's result, and two anonymous referees for a careful reading and comments.
\end{ack}

\end{spacing}


\begin{thebibliography}{{\c C}WW95}

\bibitem[ACI10]{free:resolutions:over:Koszul:algebras}
L.~Avramov, A.~Conca, and S.~Iyengar.
\newblock Free resolutions over commutative Koszul algebras.
\newblock {\em Math. Res. Lett.} 17 (2010), no. 2, 197--210. 

\bibitem[BHI17]{Koszul:algebras:defined:by:3:quadrics}
A.~Boocher, S.~H.~Hassanzadeh, and S.~Iyengar.
\newblock Koszul algebras defined by three relations.
\newblock  {\em Homological and Computational Methods in Commutative Algebra}, Springer INdAM Series 20, (2017), 53--68.

\bibitem[BH93]{Bruns:Herzog}
W.~Bruns and J.~Herzog.
\newblock {\em Cohen-Macaulay rings.} 
\newblock Cambridge Studies in Advanced Mathematics, 39. 
\newblock Cambridge University Press, Cambridge, 1993.

\bibitem[Cav00]{Caviglia} 
G.~Caviglia. 
\newblock On a class of Gorenstein algebras that are Koszul.
\newblock {\em Laurea in Mathematics (Masters thesis)},
\newblock Universit\`a di Genova, 2000. 

\bibitem[Con14]{Koszul:algebras:and:their:syzygies} 
A.~Conca.
\newblock Koszul algebras and their syzygies. 
\newblock {\em Combinatorial algebraic geometry}, 1--31, 
\newblock Lecture Notes in Math., 2108, Fond. CIME/CIME Found. Subser., Springer, Cham, 2014. 

\bibitem[CRV01]{Grobner:flags:and:Gorenstein:algebras}
A.~Conca, M.~E.~Rossi, G.~Valla.
\newblock Gr\"obner flags and Gorenstein algebras.
\newblock {\em Compositio Math.} 129 (2001), no. 1, 95--121. 

\bibitem[DGO85]{Gorenstein:algebras:and:Cayley-Bacharach}
E.~Davis, A.~Geramita, and F.~Orecchia.
\newblock Gorenstein algebras and the Cayley-Bacharach theorem. 
\newblock {\em Proc. Amer. Math. Soc.} 93 (1985), no. 4, 593--597. 

\bibitem[Eis95]{Eisenbud}
D.~Eisenbud.
\newblock {\em Commutative algebra. With a view toward algebraic geometry.} 
\newblock Graduate Texts in Mathematics, 150. 
\newblock Springer-Verlag, New York, 1995.

\bibitem[EK17]{El:Khoury:Kustin}
S.~El Khoury and A.~Kustin.
\newblock Use a Macaulay inverse system to detect an embedded deformation.
\newblock Preprint.

\bibitem[Fr\"o99]{Froberg:Koszul:algebras:survey}
R.~Fr\"oberg.
\newblock Koszul algebras. 
\newblock {\em Advances in commutative ring theory (Fez, 1997)}, 337--350, 
\newblock Lecture Notes in Pure and Appl. Math., 205, Dekker, New York, 1999. 

\bibitem[M2]{Macaulay2}
D.~Grayson and M.~Stillman.
\newblock Macaulay2, a software system for research in algebraic geometry. 
\newblock Available at \url{http://www.math.uiuc.edu/Macaulay2/}.

\bibitem[HM+13]{Lefschetz:properties}
T.~Harima, T.~Maeno, H.~Morita, Y.~Numata, A.~Wachi, and J.~Watanabe.
\newblock {\em The Lefschetz properties.} 
\newblock Lecture Notes in Mathematics, 2080. Springer, Heidelberg, 2013.

\bibitem[HM+07]{minimal:homogeneous:linkage}
C.~Huneke, J.~Migliore, U.~Nagel, and B.~Ulrich.
\newblock Minimal homogeneous liaison and licci ideals. {\em Algebra, geometry and their interactions}, 129--139, 
\newblock Contemp. Math., 448, Amer. Math. Soc., Providence, RI, 2007. 

\bibitem[HUV96]{Huneke-Ulrich:ideals}
C.~Huneke, B.~Ulrich, and W.~Vasconcelos.
\newblock On the structure of Gorenstein ideals of deviation two. 
\newblock {\em Results Math.} 29 (1996), no. 1-2, 90--99. 

\bibitem[Mas18]{Koszul:ACI's}
M.~Mastroeni.
\newblock Koszul almost complete intersections.
\newblock {\em J. Algebra} 501 (2018), 285--302. 

\bibitem[MSS19]{QGRKP1}
M.~Mastroeni, H.~Schenck, M.~Stillman.
\newblock Quadratic Gorenstein rings and the Koszul property I.
\newblock {\em Transactions of the A.M.S.}, 374, (2021), 1077--1093.

\bibitem[Mat17]{Matsuda} 
K.~Matsuda.
\newblock Non-Koszul quadratic Gorenstein toric rings. 
\newblock {\em Mathematica Scandinavica}, 123 (2018), no. 2, 161--173.

\bibitem[MS20]{jason} 
J.~McCullough, A.~Seceleanu.
\newblock Quadratic Gorenstein algebras with many surprising properties. 
\newblock {\em Arch. Math. (Basel)}, 115 (2020), no. 5, 509--521.

\bibitem[Mig98]{Migliore:linkage:book}
J.~Migliore.
\newblock {\em Introduction to liaison theory and deficiency modules.}
\newblock Progress in Mathematics, 165. 
\newblock Birkh\"auser Boston, Inc., Boston, MA, 1998.
             
\bibitem[MN13]{Gorenstein:algebras:presented:by:quadrics}
J.~Migliore and U.~Nagel.
\newblock Gorenstein algebras presented by quadrics.
\newblock {\em Collect. Math.} 64 (2013), no. 2, 211--233.

\bibitem[Pol95]{Polishchuk} 
A.~Polishchuk.
\newblock On the Koszul property of the homogeneous coordinate ring of a curve. 
\newblock {\em J. Algebra} 178 (1995), no. 1, 122--135. 

\bibitem[Tat57]{Tate}
 J.~Tate,
\newblock Homology of local and Noetherian rings. 
\newblock {\em Illinois J. Math.} 1, (1957), 14--27.
 
\bibitem[Vas94]{blowup:algebras}
W.~Vasconcelos.
\newblock {\em Arithmetic of blowup algebras.}
\newblock London Mathematical Society Lecture Note Series, 195. 
\newblock Cambridge University Press, Cambridge, 1994.
 
\bibitem[VF93]{Vishik:Finkelberg}
A.~Vishik and M.~Finkelberg.
\newblock The coordinate ring of general curve of genus $g \geq 5$ is Koszul. 
\newblock {\em J. Algebra} 162 (1993), no. 2, 535--539. 
 
\end{thebibliography}
\end{document}